\documentclass[english,12pt,oneside]{amsproc}
\usepackage[T2A]{fontenc}
\usepackage[english]{babel}
\usepackage{sseq}
\usepackage{graphics}
\usepackage{amsfonts, amssymb, amscd, amsmath}
\usepackage{latexsym}
\usepackage[matrix,arrow,curve]{xy}
\usepackage{mathabx,mathtools}
\usepackage{color}
\usepackage{mathrsfs}
\usepackage{mathdots}

\DeclareMathOperator{\lk}{lk} \DeclareMathOperator{\cone}{Cone}

 \DeclareMathOperator{\st}{st}
\DeclareMathOperator{\ver}{Vert} 
 \DeclareMathOperator{\fac}{Fac}

\newcommand{\Zo}{\mathbb{Z}}
\newcommand{\Ro}{\mathbb{R}}
\newcommand{\Co}{\mathbb{C}}

\newcommand{\ko}{\Bbbk}

\newcommand{\eqd}{\stackrel{\text{\tiny def}}{=}}

\newcommand{\minel}{\hat{0}}
\newcommand{\opst}{\st^{\circ}}

\newcommand{\ca}[1]{\mathcal{#1}}
\newcommand{\cah}[1]{\widehat{\mathcal{#1}}}

\newcommand{\sta}[1]{(\ast_{#1})}
\newcommand{\inc}[2]{[#1:#2]}

\newcommand{\E}[1]{(E_{#1})}
\newcommand{\dif}[1]{(d_{#1})}
\newcommand{\prh}[1]{\widehat{\mathcal{#1}}'}
\newcommand{\new}[1]{\overline{\mathcal{#1}}}

\newcommand{\less}[1]{\stackrel{#1}{<}}

\newcommand{\ld}{\texttt{N}}

\newcommand{\conviso}{\stackrel{\cong}{\Rightarrow}}

\newcommand{\A}{\ca{A}}

\newcommand{\Cc}{C} 

\newcommand{\I}{\ca{I}}

\newcommand{\La}{\ca{L}}
\newcommand{\loc}{\ca{U}}
\newcommand{\Lah}{\widehat{\ca{L}}}
\newcommand{\Pih}{\widehat{\Pi}}

\newcommand{\Hr}{\widetilde{H}}
\newcommand{\dd}{\partial}
\newcommand{\hh}{\mathcal{H}}
\newcommand{\hht}{\underline{\mathcal{H}}}

\newcommand{\dtot}{d_{Tot}}

\newcommand{\ST}[1]{\mbox{\upshape\scriptsize #1}}
\newcommand{\cat}{\ST{CAT}}

\newcommand{\MOD}{\ST{MOD}}

\newcounter{stmcounter}[section]

\numberwithin{equation}{section}

\theoremstyle{plain}
\newtheorem{cor}[stmcounter]{Corollary}
\newtheorem{stm}[stmcounter]{Statement}
\newtheorem{thm}[stmcounter]{Theorem}

\newtheorem{lemma}[stmcounter]{Lemma}
\newtheorem{defin}[stmcounter]{Definition}

\theoremstyle{definition}

\newtheorem{ex}[stmcounter]{Example}
\newtheorem{rem}[stmcounter]{Remark}
\newtheorem{con}[stmcounter]{Construction}

\begin{document}

\title[Sheaves over Buchsbaum posets]{Locally standard torus actions and sheaves over Buchsbaum posets}

\author[Anton Ayzenberg]{Anton Ayzenberg}
\address{Department of Mathematics, Osaka City University, Sumiyoshi-ku, Osaka 558-8585, Japan.}
\email{ayzenberga@gmail.com}

\date{\today}
\thanks{The author is supported by the JSPS postdoctoral fellowship program.}
\subjclass[2010]{Primary 57N65, 55R20; Secondary 55R91, 18F20,
55N30, 55U30, 18G40} \keywords{locally standard action, orbit type
filtration, characteristic function, manifold with corners,
simplicial poset, coskeleton filtration, sheaf over poset,
Zeeman--McCrory spectral sequence}

\begin{abstract}
We consider a sheaf of exterior algebras on a simplicial poset $S$
and introduce a notion of homological characteristic function. Two
objects are associated with these data: a graded sheaf
$\mathcal{I}$ and a graded cosheaf $\widehat{\Pi}$. When $S$ is a
homology manifold, we prove the isomorphism
$H^{n-1-p}(S;\mathcal{I})\cong H_{p}(S;\widehat{\Pi})$ which can
be considered as an extension of the Poincare duality. In general,
there is a spectral sequence $E^2_{p,q}\cong
H^{n-1-p}(S;\mathcal{U}_{n-1+q}\otimes \mathcal{I})\Rightarrow
H_{p+q}(S;\widehat{\Pi})$, where $\mathcal{U}_*$ is the local
homology stack on $S$. This spectral sequence, in turn, extends
Zeeman's spectral sequence in interpretation of McCrory. We apply
these results to toric topology. Let $X$ be an orientable manifold
with locally standard action of a compact torus and acyclic proper
faces of the orbit space. A principal torus bundle $Y$ is
associated with $X$ and the orbit type filtration on $X$ is
covered by a topological filtration on $Y$. Then the second pages
of homological spectral sequences associated with these two
filtrations are isomorphic in many positions.
\end{abstract}

\maketitle

%
%
%
%
%
%
%

\section{Introduction}\label{SecIntro}

An action of a compact torus $T^n$ on a smooth compact manifold
$M$ of dimension $2n$ is called locally standard if it is locally
modeled by the standard representation of $T^n$ on $\Co^n$. The
orbit space of a local chart is isomorphic to a nonnegative cone
$\Co^n/T^n\cong\{(x_1,\ldots,x_n)\in \Ro^n\mid x_i\geqslant 0\}$,
thus the orbit space $Q=M/T^n$ of the whole manifold has a natural
structure of manifold with corners. Points from interiors of
$k$-dimensional faces of $Q$ are the $k$-dimensional orbits of the
action. For any face $G$ of $Q$ consider the stabilizer subgroup
$T_G\subset T^n$ of points in the interior of $G$. The mapping
sending the face $G$ to the toric subgroup $T_G$ is called
characteristic data.

For any manifold $M$ with locally standard torus action having the
orbit space $Q$ there exists a principal $T^n$-bundle $Y\to Q$
such that $M$ is equivariantly homeomorphic to the identification
space $X=Y/\sim$, where $\sim$ identifies points over a face
$G\subset Q$ differing by the action of $T_G$ \cite{Yo}. Thus any
manifold with locally standard action is uniquely determined, up
to equivariant homeomorphism, by three objects: a manifold with
corners $Q$, a principal torus bundle $Y$ over $Q$ (these bundles
are encoded by their Euler classes lying in $H^2(Q;\Zo^n)$), and
characteristic data.

For a manifold with corners $Q$ consider the dual poset $S_Q$. The
elements of $S_Q$ are the faces of $Q$ and the order is given by
reversed inclusion. If $Q$ is the orbit space of a manifold with
locally standard action, then $S_Q$ is a simplicial poset (see
Definition~\ref{definSimpPoset}).

The description of topology of $X$ in terms of the combinatorial
data is difficult and, in general, far from being accomplished.
The cohomology and equivariant cohomology rings are unknown and
even Betti numbers haven't been explicitly calculated yet.

Nevertheless, there are several important particular cases which
are known and well studied. If the orbit space $Q$ is isomorphic
to a simple polytope, the manifold $X$ is called quasitoric. This
particular case was introduced and studied in the seminal work of
Davis and Januszkiewicz \cite{DJ} and underlied the development of
toric topology. Quasitoric manifolds are natural topological
generalizations of smooth projective toric varieties. The reason
which makes quasitoric manifolds feasible from topological
viewpoint is that the orbit space has trivial topology (the
convexity happens to be not so important).

This setting may be generalized to the case when all faces of $Q$
are acyclic. This situation is very close to toric varieties or
quasitoric manifolds and the answer is also very similar
\cite{MasPan}:
\[
H^*_{T^n}(X;\Zo)\cong \Zo[S_Q];\qquad H^*(X;\Zo)\cong
\Zo[S_Q]/(\theta_1,\ldots,\theta_n),
\]
where $\Zo[S_Q]$ is the face ring of the simplicial poset $S_Q$
and $(\theta_1,\ldots,\theta_n)$ is a regular sequence of degree
$2$ in $\Zo[S_Q]$, determined by the characteristic data.

There are several papers where the calculation of topological
invariants was performed for more general examples. In \cite{AMPZ}
we proved that whenever all proper faces of $Q$ are acyclic and
$Y\to Q$ is a trivial bundle, the equivariant cohomology ring is
represented as a direct sum (as rings and as modules over
$H^*(BT^n;\Zo)$):
\[
H^*_{T^n}(X;\Zo)\cong \Zo[S_Q]\oplus H^*(Q;\Zo).
\]
We also calculated Betti numbers and partly described the ring
structure of $H^*(X,\Zo)$ when $X$ is an orientable toric origami
manifold with acyclic proper faces of the orbit space. This is a
very restricted class of manifolds with locally standard actions,
but even in this case many interesting phenomena sprang up. Betti
numbers of 4-dimensional toric origami manifolds without any
restrictions on proper faces were calculated in \cite{HoPi}. The
cohomology rings of 4-dimensional manifolds whose orbit spaces are
polygons with holes were described in \cite{PodSar}.

In \cite{Yo} Yoshida introduced the cohomological spectral
sequence converging to $H^*(X;\Zo)$ for any $X$, but generally
this spectral sequence does not quickly collapse, so it is
difficult to extract any explicit information, such as Betti
numbers, from it. This approach requires an extra effort to obtain
a concrete result. However, for some particular choices of $X$
this extra effort can be done.

In this paper we study the homological structure of a manifold
$X$, and related objects, using the filtration of $X$ by orbit
types
\begin{equation}\label{eqIntroFiltrX}
X_0\subset X_1\subset\ldots\subset X_n=X.
\end{equation}
Here $X_i$ is the union of all $T^n$-orbits of dimension at most
$i$, so $\dim X_i = 2i$. This filtration induces a spectral
sequence $\E{X}^r_{p,q}\Rightarrow H_{p+q}(X)$, where
$\E{X}^1_{p,q}\cong H_{p+q}(X_p,X_{p-1})$. By dimensional reasons,
$\E{X}^r_{p,q}=0$ for $p<q$ and $r\geqslant 1$.

There is a natural topological filtration of $Y$ which covers the
orbit type filtration of $X$, and the map $f\colon Y\to X$ induces
the map of homological spectral sequences
\begin{equation}\label{eqIntroMapSpecSeq}
f^r_*\colon \E{Y}^r_{p,q}\to\E{X}^r_{p,q}.
\end{equation}
If the proper faces of $Q$ are acyclic, we prove that the map
$f^2_*$ is an isomorphism for $p>q$. Thus every entry of
$\E{X}^r_{p,q}$ away from the diagonal is known, at least if the
structure of $\E{Y}^2_{p,q}$ is known.


To prove the above-mentioned isomorphism (Theorem
\ref{thmToricSpecSec}), we place the maps $f^2_*\colon
\E{Y}^2_{p,q}\to\E{X}^2_{p,q}$ into a long exact sequence and show
that certain intermediate terms of this sequence vanish. These
intermediate terms are the cohomology modules $H^*(S_Q;\I)$ of a
graded sheaf $\I$ on $S_Q$, whose values are the ideals in the
homology algebra $H_*(T^n)$ generated by the vector subspaces
$H_1(T_G)\subset H_1(T^n)$. The vanishing of these sheaf
cohomology in certain degrees is the most nontrivial and essential
part of the work. It follows from the duality:
\begin{equation}\label{eqIntroDualityManif}
H^{n-1-i}(S_Q;\I)\cong H_{i}(S_Q;\Pih),
\end{equation}
(Theorem \ref{thmDuality}) which holds for the homology manifold
$S_Q$ and extends the Poincare duality $H^{n-1-i}(S_Q;\ko)\cong
H_{i}(S_Q;\ko)$. Here $\Pih$ is a cellular cosheaf on $S_Q$ whose
value on a face $G\subset Q$ is the ideal in $H_*(T^n)$ generated
by the volume form of the submodule $H_*(T_G)\subset H_*(T^n)$.

We study this duality in a broader and quite natural setting. For
a simplicial poset $S$ there exists a Zeeman--McCrory spectral
sequence $\E{ZM}^r_{p,q}$. It converges to the homology of $S$,
and its second page is the cohomology of local homology stacks
$\loc_*$ on $S$. If $S$ is a manifold, this sequence collapses at
a second page and gives a standard proof of the Poincare duality.
Thus Zeeman--McCrory spectral sequence can be roughly considered
as a generalization of Poincare duality to non-manifolds.

We prove that there is a spectral sequence, starting with
$H^*(S;\loc_*\otimes \I)$ and converging to $H_*(S;\Pih)$ (Theorem
\ref{thmSpecSec}). For homology manifolds it collapses to the
isomorphism \eqref{eqIntroDualityManif}.

The work is organized as follows. In Section \ref{SecSheaves} we
review the basic notions: simplicial posets, sheaves, cosheaves,
and Zeeman--McCrory spectral sequence. The word ``sheaf'' will
always mean ``a sheaf over a finite poset''. It is not used in its
broadest topological sense, but rather replaces the term stack or
local coefficient system. In Section \ref{SecAlgebras} we
introduce the notion of homological characteristic function,
define two objects associated with this object: the sheaf $\I$,
and the cosheaf $\Pih$, and formulate Theorems \ref{thmSpecSec}
and \ref{thmDuality}, proving the duality
\eqref{eqIntroDualityManif}. Theorem \ref{thmSpecSec} is proved in
Section \ref{SecProof12}, and Theorem \ref{thmDuality} follows as
its particular case. Preliminaries on manifolds with locally
standard actions are given in Section \ref{SecManifolds}. We
introduce topological filtrations on $Q$, $X$, and $Y$, and
formulate Theorem \ref{thmToricSpecSec}, which states that modules
$\E{X}^r_{p,q}$ are isomorphic to $\E{Y}^2_{p,q}$ for $p>q$.
Section \ref{SecProof3} is devoted to the proof of Theorem
\ref{thmToricSpecSec}; there we explain the connection of
manifolds with torus actions and the sheaf-theoretical part of the
work.

%
%
%
%
%
%
%

\section{Sheaves and cosheaves over simplicial posets}\label{SecSheaves}

%
%

\subsection{Preliminaries on simplicial posets}

\begin{defin}\label{definSimpPoset}
A finite partially ordered set (poset) is called simplicial if
there exists a minimal element $\minel\in S$ and, for any $I\in
S$, the lower order ideal $\{J\in S\mid J\leqslant I\}$ is
isomorphic to the boolean lattice $2^{[k]}$ (the poset of faces of
a $(k-1)$-dimensional simplex) for some $k\geqslant 0$.
\end{defin}

The elements of $S$ are called \emph{simplices}. The number $k$ in
the definition is denoted by $|I|$ and called the rank of a
simplex $I$. Also set $\dim I = |I|-1$. A simplex of rank $1$ is
called a \emph{vertex}; the set of all vertices is denoted by
$\ver(S)$. A subset $L\subset S$ closed under taking sub-simplices
is called a simplicial subposet.

The notation $I\less{i}J$ is used whenever $I\leqslant J$ and
$|J|-|I|=i$. If $S$ is a simplicial poset, then for each
$I\less{2}J\in S$, there exist exactly two simplices $J'\neq J''$
between $I$ and $J$:
\begin{equation}\label{eqSquareCond}
I\less{1} J',J'' \less{1} J.
\end{equation}
For simplicial poset $S$ a ``sign convention'' can be chosen. It
means that we can associate an incidence number $\inc{J}{I}=\pm 1$
with any pair $I\less{1}J\in S$ such that
\begin{equation}\label{eqSignSquare}
\inc{J}{J'}\cdot \inc{J'}{I}+\inc{J}{J''}\cdot \inc{J''}{I}=0
\end{equation}
for any combination \eqref{eqSquareCond}. The choice of sign
convention is the same as orienting each simplex in $S$. We fix an
arbitrary sign convention and use it in the following
considerations.

Notice that the set of simplices of any finite simplicial complex
obviously forms a simplicial poset. Thus the notion of simplicial
poset is a straightforward generalization of abstract simplicial
complex.

For $I\in S$ consider the following subset of $S$:
\[
\opst_SI=\{J\in S\mid J\geqslant I\},
\]
called the \emph{open star} of $I$. It is easily seen that
$S\setminus\opst_SI$ is a simplicial subposet of $S$.

We also define the \emph{link} of a simplex $I\in S$:
\[
\lk_SI=\{J\in S\mid J\geqslant I\}.
\]
This set inherits the order relation from $S$, and $\lk_SI$ is a
simplicial poset with respect to this order, with the minimal
element $I$. The reason why we used to different notation for the
same thing is that it is convenient to distinguish between
$\opst_SI$, which is considered as a subset of $S$ (but not a
subposet!), and $\lk_SI$, which is considered as a simplicial
poset on its own (and which is, in general, not included in $S$ as
a subposet in any meaningful way). Note that $\lk_S\minel=S$.

Let $S'$ be the barycentric subdivision of $S$. By definition,
$S'$ is a simplicial complex on the set $S\setminus \minel$ whose
simplices are the chains of elements of $S$. By definition, the
geometric realization of $S$ is the geometric realization of its
barycentric subdivision $|S|\eqd|S'|$. One can also think of $|S|$
as a CW-complex with simplicial cells. Such topological models of
simplicial posets were called \emph{simplicial cell complexes} and
were studied in \cite{BPposets}.

A poset $S$ is called \emph{pure} if all its maximal elements have
equal dimensions. A poset $S$ is pure whenever $S'$ is pure.

In the following $\ko$ denotes the ground ring; it may be either a
field or the ring of integers. The (co)homology of simplicial
poset $S$ mean the (co)homology of its geometrical realization
$|S|$. If the coefficient ring in the notation of (co)homology is
omitted, it is supposed to be $\ko$.

\begin{defin}\label{definBuchCMposets}
Simplicial complex $K$ of dimension $n-1$ is called Buchsbaum
(over $\ko$) if $\Hr_i(\lk_KI;\ko)=0$ for all $\minel\neq I\in K$
and $i\neq n-1-|I|$. If $K$ is Buchsbaum and, moreover,
$\Hr_i(K;\ko)=0$ for $i\neq n-1$ then $K$ is called
Cohen--Macaulay.

Simplicial poset $S$ is called Buchsbaum (resp. Cohen--Macaulay)
if $S'$ is a Buchsbaum (resp. Cohen--Macaulay) simplicial complex.
\end{defin}

\begin{rem}\label{remBuchPoset}
By \cite[Sec.6]{NS}, $S$ is Buchsbaum whenever
$\Hr_i(\lk_SI;\ko)=0$ for all $\minel\neq I\in S$ and $i\neq
n-1-|I|$. Similarly, $S$ is Cohen--Macaulay whenever
$\Hr_i(\lk_SI;\ko)=0$ for all $I\in S$ and $i\neq n-1-|I|$.
\end{rem}

Typical examples of Buchsbaum posets are triangulations (and, more
generally, simplicial cell decompositions) of manifolds. Typical
examples of Cohen--Macaulay posets are triangulations of spheres.
A poset $S$ is Buchsbaum whenever all its proper links are
Cohen--Macaulay.

One can easily check that whenever $S$ is Buchsbaum and connected,
then $S$ is pure. In the following only pure simplicial posets are
considered.

%
%

\subsection{Cellular sheaves}
Let $\MOD_{\ko}$ be the category of $\ko$-modules. The notation
$\dim V$ is used for the rank of a $\ko$-module $V$.

Each simplicial poset $S$ determines a small category $\cat(S)$
whose objects are the elements of $S$ and the morphisms are the
inequalities $I\leqslant J$. A \emph{cellular sheaf} \cite{Curry}
(or a stack \cite{McCrory}, or a local coefficient system
elsewhere) on $S$ is a covariant functor $\A\colon \cat(S)\to
\MOD_{\ko}$. We simply call $\A$ a sheaf on $S$ and hope this will
not lead to a confusion, since other meanings of this word do not
appear in the paper. The maps $\A(J_1\leqslant J_2)$ are called
\emph{restriction maps}. The cochain complex $(\Cc^*(S;\A),d)$ is
defined as follows:
\[
\Cc^*(S;\A) = \bigoplus_{i\geqslant -1}\Cc^i(S;\A),\qquad
\Cc^i(S;\A) = \bigoplus_{\dim I=i}\A(I),
\]
\[
d\colon\Cc^i(S;\A)\to \Cc^{i+1}(S;\A), \qquad
d=\bigoplus_{I\less{1}I', \dim I=i}[I':I]\A(I\leqslant I').
\]
The sign convention \eqref{eqSignSquare} implies that $d^2=0$.
Thus $(\Cc^*(S;\A),d)$ is a differential complex. Define the
cohomology of $\A$ as the cohomology of this complex:
\begin{equation}\label{eqSheafCohomDef}
H^*(S;\A)\eqd H^*(\Cc^*(S;\A),d).
\end{equation}

\begin{rem}
Cohomology of $\A$ defined this way coincide with any other
meaningful definition of cohomology. For example the derived
functors of the functor of global sections give the same groups as
\eqref{eqSheafCohomDef} (refer to \cite{Curry} for a broad
exposition of this subject).
\end{rem}

A sheaf $\A$ on $S$ can be restricted to a simplicial subposet
$L\subset S$. The complexes $(\Cc^*(L,\A),d)$ and
$(\Cc^*(S;\A)/\Cc^*(L;\A),d)$ are defined as usual. The latter
complex gives rise to a relative version of sheaf cohomology:
$H^*(S,L;\A)$.

\begin{rem}\label{remTruncateTheSheaf}
It is standard in topological literature to consider cellular
sheaves which do not take values on $\minel\in S$, since in
general this element does not have a geometrical meaning. However,
this extra value $\A(\minel)$ will be important in the
considerations of Section \ref{SecProof3}. Therefore the
cohomology group may be nontrivial in degree $-1=\dim\minel$. If a
sheaf $\A$ is defined on $S$, then we can consider its truncated
version $\underline{\A}$ which coincides with $\A$ on
$S\setminus\{\minel\}$ and vanishes on $\minel$.
\end{rem}

The notions of maps, (co)kernels, (co)images, tensor products of
sheaves over $S$ are defined in an obvious componentwise manner.
For example, if $\ca{A}$ and $\ca{B}$ are two sheaves on $S$, then
$\ca{A}\otimes\ca{B}$ is a sheaf on $S$ with values
$(\ca{A}\otimes\ca{B})(I)=\ca{A}(I)\otimes\ca{B}(I)$ and
restriction maps $(\ca{A}\otimes\ca{B})(I\leqslant
J)=\ca{A}(I\leqslant J)\otimes\ca{B}(I\leqslant J)$. In the realm
of finite simplicial posets the distinction between ``sheaves''
and ``presheaves'' vanishes, which makes things simpler than they
are in algebraic geometry.

\begin{ex}\label{exSheafConstant}
Let $W$ be a $\ko$-module. By abuse of notation let $W$ denote the
globally constant sheaf on $S$. It takes the constant value $W$ on
$I\neq \minel$ and vanishes on $\minel$. All nontrivial
restriction maps are identity isomorphisms. If $W$ is
torsion-free, we have $H^*(S;W)\cong H^*(S;\ko)\otimes W$ by the
universal coefficients formula.
\end{ex}

\begin{ex}\label{exLocConstSheaf}
A locally constant sheaf valued by $W\in\MOD_\ko$ is a sheaf
$\ca{W}$ which satisfies $\ca{W}(\minel)=0$, $\ca{W}(I)\cong W$
for $I\neq\minel$ and all nontrivial restriction maps are
isomorphisms (but may be not identity isomorphisms).
\end{ex}



\begin{ex}\label{exSheafLocHomol}
Following \cite{McCrory}, define $i$-th local homology sheaf
$\loc_i$ on $S$ by setting $\loc_i(\minel)=0$ and
\begin{equation}\label{eqDefLocHomSheaf}
\loc_i(J)=H_i(S,S\setminus\opst_SJ;\ko)
\end{equation}
for $J\neq \minel$. The restriction maps $\loc_i(J_1<J_2)$ are
induced by inclusions of subsets
$\opst_SJ_2\hookrightarrow\opst_SJ_1$. Standard topological
arguments imply that a simplicial poset $S$ is Buchsbaum if and
only if $\loc_i=0$ for $i\neq n-1$ (see also Remark
\ref{remLocalSheavesOnPoset} below).
\end{ex}

\begin{defin}
Buchsbaum simplicial poset $S$ is called homology manifold
(orientable over $\ko$) if its local homology sheaf $\loc_{n-1}$
is isomorphic to the constant sheaf~$\ko$.
\end{defin}

$S$ is an orientable homology manifold if and only if its
geometrical realization is an orientable homology manifold in a
usual topological sense.

%
%

\subsection{Cosheaves}

A cellular cosheaf \cite{Curry} is a contravariant functor
$\cah{A}\colon \cat(S)^{op}\to \MOD_{\ko}$. The homology of a
cosheaf are defined similar to the cohomology of a sheaf:
\begin{gather*}
\Cc_*(S;\cah{A}) = \bigoplus_{i\geqslant -1}\Cc_i(S;\cah{A})\quad
\Cc_i(S;\cah{A}) = \bigoplus_{\dim I=i}\A(I)\\
d\colon\Cc_i(S;\cah{A})\to \Cc_{i-1}(S;\cah{A}),\quad
d=\bigoplus_{I>_1I', \dim I=i}[I:I']\cah{A}(I\geqslant I'),\\
H_*(S;\cah{A})\eqd H_*(\Cc_*(S;\cah{A}),d).
\end{gather*}

\begin{ex}\label{exSheafCosheaf}
Each locally constant sheaf $\ca{W}$ on $S$ determines the locally
constant cosheaf $\cah{W}$ by inverting all maps, i.e.
$\cah{W}(I)\cong \ca{W}(I)$ and $\cah{W}(I>J)=(\ca{W}(J<I))^{-1}$.
\end{ex}

\begin{rem}
Notice that the notation $H_*(S;\ko)$ can mean either the homology
of the geometric realization $|S|$ or the homology of a globally
constant cosheaf $\ko$ on $S$. Obviously these two meanings are
consistent, and the same for cohomology of a constant sheaf.
\end{rem}

%
%

\subsection{Coskeleton filtration and dual faces}

In the following we suppose that $S$ is pure and $\dim S = n-1$.

\begin{con}
Let us recall the construction of \emph{coskeleton filtration} on
$|S|$. Consider the barycentric subdivision $S'$ of the pure
simplicial poset $S$. By definition, $S'$ is a simplicial complex
on the set $S\setminus \minel$ and $k$-simplices of $S'$ have the
form $(I_0< I_1<\ldots<I_k)$, where $I_i\in S\setminus\minel$. For
each $I\in S\setminus\{\minel\}$ consider the subcomplex of the
barycentric subdivision:
\[
G_I=\{(I_0< I_1<\ldots)\in S'\mbox{ such that } I_0\geqslant
I\}\subset S',
\]
and the subsets
\[
\dd G_I=\{(I_0< I_1<\ldots)\in S'\mbox{ such that } I_0>I\}
\subset S',\mbox{ and }\quad G_I^{\circ}=G_I\setminus\dd G_I.
\]

It is easily seen that $\dim G_I=n-1-\dim I$ since $S$ is pure. We
have $G_I\subset G_J$ whenever $J<I$. The complex $G_I$ (or its
geometrical realization $|G_I|$) is called the \emph{face} or the
\emph{pseudocell} of $|S|$ dual to $I\in S$. The boundary $\dd
G_I$ of a face $G_I$ is the union of some faces of smaller
dimensions.

Let $S_i=\bigcup_{\dim G_I\leqslant i}G_I$ for $-1\leqslant
i\leqslant n-1$. Thus $S_i$ is a simplicial subcomplex of $S'$.
The filtration
\begin{equation}\label{eqCoskeletonSimp}
\emptyset=S_{-1}\subset S_0\subset S_1\subset\ldots\subset
S_{n-1}= S',
\end{equation}
and the corresponding topological filtration
\begin{equation}\label{eqCoskeleton}
\emptyset=|S_{-1}|\subset |S_0|\subset |S_1|\subset\ldots\subset
|S_{n-1}|= |S|,
\end{equation}
are called the coskeleton filtrations of $S'$ and $|S|$
respectively \cite{McCrory}.
\end{con}

For a pair $I\less{1}J\in S$ consider the map:
\begin{multline}\label{eqMattachingMap}
m^q_{I,J}\colon H_{q+\dim G_I}(G_I,\dd G_I)\to H_{q+\dim
G_I-1}(\dd G_I)\to \\ \to H_{q+\dim G_I-1}(\dd G_I, \dd
G_I\setminus G_J^{\circ})\cong H_{q+\dim G_J}(G_J, \dd G_J),
\end{multline}
where the first map is the connecting homomorphism in the long
exact sequence of homology for the pair $(G_I,\dd G_I)$, and the
last isomorphism is due to excision. The homology spectral
sequence associated with filtration \eqref{eqCoskeleton} runs
\[
\E{S}^1_{p,q}=H_{p+q}(S_p,S_{p-1})\Rightarrow H_{p+q}(S).
\]
The first differential $\dif{S}^1$ is the sum of the maps
$m^q_{I,J}$ over all pairs $I\less{1}J$, $I,J\in S$.

\begin{con}\label{conSheafOnS}
Given a sign convention on $S$, for each $q$ consider the sheaf
$\hh_q$ on $S$ given by
\[
\hh_q(I)=H_{q+\dim G_I}(G_I,\dd G_I)
\]
for $I\neq\minel$, and $\hh_q(\minel)=0$. For neighboring
simplices $I\less{1}J$ define the restriction map as
$\hh_q(I\less{1}J) = \inc{J}{I} m^q_{I,J}$. For general
$I\less{k}J$ consider any saturated chain in $S$ between $I$ and
$J$:
\[
I\less{1}J_1\less{1}\ldots\less{1}J_{k-1}\less{1}J,
\]
and set
\[
\hh_q(I<J)\eqd\hh_q(J_{k-1}<J)\circ\ldots\circ\hh_q(I<J_1).
\]
\end{con}

\begin{lemma}
The map $\hh_q(I<J)$ thus defined does not depend on a choice of
saturated chain between $I$ and $J$.
\end{lemma}

\begin{proof}
The differential $\dif{S}^1$ satisfies $(\dif{S}^1)^2=0$, thus
$m^q_{J',J}\circ m^q_{I,J'}+m^q_{J'',J}\circ m^q_{I,J''}=0$. By
combining this with \eqref{eqSignSquare} we see that $\hh_q(I<J)$
is independent of the chain if its length is $2$. In general,
since $\{T\mid I\leqslant T\leqslant J\}$ is a boolean lattice,
any two saturated chains between $I$ and $J$ are connected by a
sequence of elementary flips $[J_k~\less{1}~T_1~\less{1}~J_{k+2}]
\rightsquigarrow [J_k~\less{1}~T_2~\less{1}~J_{k+2}]$ and the
statement follows.
\end{proof}

Thus the sheaves $\hh_q$ are well defined. They will be called the
structure sheaves of $S$. From the definition of a cochain complex
directly follows

\begin{cor}
The cochain complexes of structure sheaves coincide with
$\E{S}^1_{*,*}$ up to change of indices:
\[
(\E{S}^1_{*,q},\dif{S}^1)\cong(C^{n-1-*}(\hh_q),d).
\]
\end{cor}

\begin{rem}\label{remLocalSheavesOnPoset}
There exists an isomorphism of sheaves
\begin{equation}\label{eqLocStrIsom}
\hh_q\cong \loc_{q+n-1},
\end{equation}
where $\loc_*$ are the sheaves of local homology defined in
Example \ref{exSheafLocHomol}. Indeed, it can be shown that
$H_i(S,S\setminus\opst_SI)\cong H_{i-\dim I}(G_I,\dd G_I)$ and
these isomorphisms can be chosen compatible with restriction maps.
For simplicial complexes this fact is proved in
\cite[Sec.6.1]{McCrory}; the case of simplicial posets is rather
similar. Note that the definition of $\hh_*$ depends on the sign
convention while $\loc_*$ does not. This makes no contradiction
since the isomorphism \eqref{eqLocStrIsom} itself depends on the
choice of orientations.

The isomorphism \eqref{eqLocStrIsom} implies that $S$ is Buchsbaum
if and only if $\hh_q=0$ for $q\neq 0$. Simplicial poset $S$ is an
orientable manifold if it is Buchsbaum and, moreover,
$\hh_0\cong\ko$.
\end{rem}

%
%

\subsection{Zeeman--McCrory spectral sequence}

From the considerations of the previous subsection easily follows

\begin{stm}[McCrory, \cite{McCrory}]
There exists a spectral sequence, located in fourth quadrant,
\begin{gather}\label{eqZMspecSeq}
\E{ZM}^r_{p,q},\quad d^r\colon \E{ZM}^r_{p,q}\to \E{ZM}^r_{p-r,q+r-1};\\
\E{ZM}^2_{p,q}\cong H^{n-1-p}(S;\loc_{n-1+q}) \Rightarrow
H_{p+q}(S;\ko).
\end{gather}
It is isomorphic to the homological spectral sequence, associated
with the coskeleton filtration of $|S|$.
\end{stm}

For us, however, it will be more convenient to work with structure
sheaves $\hh_*$ rather than local homology sheaves $\loc_*$. For a
Buchsbaum simplicial poset the sheaf $\hh_i$ vanish for $i\neq 0$.
Thus $\E{ZM}^2_{p,q}=0$ for $q\neq 0$ and the spectral sequence
collapses at the second page inducing the isomorphism
\[
H^{n-1-p}(S;\hh_0)\cong H_p(S;\ko).
\]
When $S$ is an orientable homology manifold, this gives a Poincare
duality isomorphism
\[
H^{n-1-p}(S;\ko)\cong H_p(S;\ko).
\]

%
%

\subsection{Corefinements of sheaves}\label{SubsecCorefin}

In this section we develop a technical notion which will be used
further in the proofs.

Let $\ca{A}$ be a sheaf on $S$. Define a cosheaf $\prh{A}$ on the
barycentric subdivision $S'$ by
\[
\prh{A}(I_1<\ldots<I_k)=\ca{A}(I_1)
\]
with corestriction maps determined naturally by restriction maps
of $\ca{A}$:
\[
\prh{A}((I_1<\ldots<I_k)\supset(J_1<\ldots<J_s))=\ca{A}(I_1\leqslant
J_1).
\]
We call $\prh{A}$ a corefinement of a sheaf $\ca{A}$. The faces
$G_I$ and their boundaries $\dd G_I$ are the simplicial
subcomplexes of $S'$ so one can restrict $\prh{A}$ to them. Next
lemma follows easily from the definitions.

\begin{lemma}\label{lemmaTermSplitGenSheaf}
\[
H_q(S_p,S_{p-1},\prh{A})\cong \bigoplus_{I, \dim G_I=p}
H_q(G_I,\dd G_I; \prh{A}).
\]
\end{lemma}

Similar to \eqref{eqMattachingMap} there is a map
\begin{multline}\label{eqMattachingMapA}
m^{q,\ca{A}}_{I,J}\colon H_{q+\dim G_I}(G_I,\dd G_I;\prh{A})\to
H_{q+\dim G_I-1}(\dd G_I;\prh{A})\to \\ \to H_{q+\dim G_I-1}(\dd
G_I, \dd G_I\setminus G_J^{\circ};\prh{A})\cong H_{q+\dim
G_J}(G_J, \dd G_J;\prh{A}).
\end{multline}
These maps allow to define new sheaves $\new{A}_q$ on $S$ by
setting $\new{A}_q(I)=H_{q+\dim G_I}(G_I,\dd G_I;\prh{A})$ with
restriction maps defined similar to Construction
\ref{conSheafOnS}.

\begin{lemma}\label{lemmaNewSheafIsOld}
If $A(I)$ is torsion-free for all $I\in S$, then there exist
natural isomorphisms
\[
H_r(G_I,\dd G_I; \prh{A})\cong H_r(G_I,\dd G_I;\ko)\otimes
\ca{A}(I).
\]
The maps $m^{q,\ca{A}}_{I,J}$ coincide with
$m^{q}_{I,J}\otimes\ca{A}(I<J)$ up to these isomorphisms. Thus the
sheaf $\new{A}_q$ is isomorphic to $\hh_q\otimes\ca{A}$.
\end{lemma}

\begin{proof}
By the definition of $\prh{A}$ we have
\[
H_r(G_I,\dd G_I; \prh{A})\cong H_r(G_I,\dd G_I; \ca{A}(I)),
\]
since the value of $\prh{A}$ on all simplices of $G_I^{\circ}$ is
exactly $\ca{A}(I)$. The rest follows from universal coefficients
formula.
\end{proof}

%
%
%
%
%
%
%

\section{Exterior algebras and characteristic
functions}\label{SecAlgebras}

Let $V$ be a free $\ko$-module of dimension $\ld$. Let
$\Lambda[V]$ denote the free exterior algebra generated by $V$,
that is the quotient of a free tensor algebra $T[V]$ by the
relations $v\otimes v=0$ for all $v\in V$. The algebra
$\Lambda[V]$ is graded by degrees of exterior forms.

\begin{defin}
Let us fix a simplicial poset $S$ and a locally constant sheaf
$\ca{V}$ on $S$. A collection of vectors $\{\omega_i\in
\ca{V}(i)\mid i\in \ver(S)\}$ is called a homological
$\ko$-characteristic function if it satisfies the following
$\sta{\ko}$-condition:

For each simplex $I\in S\setminus\minel$ whose vertices are
$i_1,\ldots,i_k$, the vectors
\[
\ca{V}(i_1\leqslant I)(\omega_{i_1}), \ldots, \ca{V}(i_k\leqslant
I)(\omega_{i_k}) \in \ca{V}(I)
\]
are linearly independent over $\ko$ and span a direct summand in
$\ca{V}(I)$.
\end{defin}

For a locally constant sheaf $\ca{V}$ on $S$, valued by the vector
space $V$, consider the sheaf $\La=\Lambda[\ca{V}]$ of graded
exterior algebras generated $\ca{V}$. This means that
$\La(I)=\Lambda[\ca{V}(I)]$, and $\La(I\leqslant J)$ is an
isomorphism of graded exterior algebras generated by the
isomorphism $\ca{V}(I\leqslant J)\colon \ca{V}(I)\to \ca{V}(J)$ in
degree one levels. Let $\Lah$ denote the locally constant cosheaf
of exterior algebras corresponding to a sheaf $\La$ (see
Example~\ref{exSheafCosheaf}).

Let $\{\omega_i\in \ca{V}(i)\mid i\in\ver(S)\}$ be a homological
characteristic function. If $i$ is a vertex of a simplex $I$, then
the restriction map $\ca{V}(i\leqslant I)$ sends the vector
$\omega_i\in \ca{V}(i)$ to some vector in $\ca{V}(I)$. By abuse of
notation we denote the target vector by the same letter
$\omega_i$. So far the definition of homological characteristic
function implies that the set
$\{\omega_{i_1},\ldots,\omega_{i_k}\}$ freely spans a direct
summand of $\ca{V}(I)$ whenever $i_1,\ldots,i_k$ are vertices of
$I$. Note, that $\La(I)$ is an exterior algebra generated by
$\ca{V}(I)$, so the vectors $\omega_i$ can be considered as linear
forms in $\La(I)$.

\begin{con}
Consider a subsheaf $\I\subset\La$, defined as follows. For a
simplex $I$ with vertices $i_1,\ldots,i_k$ we set the value of
$\I$ on $I$ to be the ideal of $\La(I)$, generated by the linear
forms:
\[
\I(I) = (\omega_{i_1}, \ldots, \omega_{i_k}).
\]
It is easily seen that whenever $I\leqslant J$, the restriction
map $\La(I\leqslant J)$ sends the ideal $\I(I)$ generated by the
smaller set of elements into the ideal $\I(J)$ generated by the
larger set of elements. Thus the restriction maps of the sheaf
$\I$ are induced from those of $\La$ and are well defined.
\end{con}

\begin{con}
Let us define another type of ideals associated with a
characteristic function.

Let $J=\{i_1,\ldots,i_k\}$ be a nonempty subset of vertices of a
simplex $I\in S$. Consider the element $\pi_J\in\La(I)=\Lah(I)$,
$\pi_J=\bigwedge_{i\in J}\omega_i$. By the definition of
characteristic function, the elements $\{\omega_i\mid i\in J\}$
are linearly independent, thus $\pi_J$ is a non-zero form of
degree $|J|$. Let $\Pi_J\subset \La(I)$ be the principal ideal
generated by $\pi_J$. The restriction maps $\La(I<I')$ (and
corestriction maps $\Lah(I'>I)=\La(I<I')^{-1}$) identify
$\Pi_J\subset \La(I)$ with $\Pi_J\subset \La(I')$.

Let us define a subcosheaf $\Pih$ of ideals in $\Lah$. If $J$ is
the whole set of vertices of a simplex $I\neq\minel$ we define
$\Pih(I)\eqd\Pi_J\subset\Lah(I)$. If $I'<I$, the corestriction map
$\Lah(I'>I)$ injects $\Pih(I')$ into $\Pih(I)$, since the form
$\pi_{I'}$ is divisible by $\pi_{I}$. Thus $\Pih$ is a
well-defined graded sub-cosheaf of $\Lah$. We formally set
$\Pih(\minel)=0$.
\end{con}

Now we can formulate our main homological results.

\begin{thm}\label{thmSpecSec}
Let $S$ be a pure simplicial poset of dimension $n-1$, and $\I$,
$\Pih$ the sheaf and cosheaf over $S$, determined by some
homological $\ko$-characteristic function. Then there exists a
spectral sequence
\[
E^2_{s,k}\cong H^{n-1-s}(S;\hh_k\otimes \I)\Rightarrow
H_{s+k}(S;\Pih),
\]
\[
d^r\colon E^r_{s,k}\to E^r_{s-r,k+r-1}
\]
which respects the inner gradings of $\I$ and $\Pih$.
\end{thm}

If $S$ is Buchsbaum, the spectral sequence of Theorem
\ref{thmSpecSec} collapses at a second page and implies

\begin{thm}\label{thmDuality}
For Buchsbaum simplicial poset $S$ of dimension $n-1$ there exists
an isomorphism $H^k(S;\hh_0\otimes\I)\cong H_{n-1-k}(S;\Pih)$
which respects the inner gradings of $\I$ and $\Pih$.
\end{thm}

\begin{cor}
If $S$ is a homology $(n-1)$-manifold, then there is an
isomorphism $H^k(S;\I)\cong H_{n-1-k}(S;\Pih)$, respecting the
inner gradings.
\end{cor}

Let $\I^{(q)},\Pih^{(q)}$ denote the homogeneous parts of inner
degree $q$ of the corresponding sheaves $\I,\Pih$.

\begin{cor}[Key corollary]\label{corKeyCorol}
If $S$ is a Buchsbaum simplicial poset, then
$H^j(S;\hh_0\otimes\I^{(q)})=0$ for $j\leqslant n-1-q$.
\end{cor}

\begin{proof}
By Theorem \ref{thmDuality}, it is sufficient to prove that
$H_j(S;\Pih^{(q)})=0$ for $j\geqslant q$. The ideal
$\Pih(I)=\Pi_I$ is generated by the element $\pi_I$ of degree
$|I|=\dim I+1$. Thus $\Pi_I^{(q)}=0$ for $q\leqslant \dim I$.
Hence the corresponding part of the chain complex vanishes, and
the homology in these degrees vanish as well.
\end{proof}

\begin{rem}
The exterior forms of the top power, $\Lambda[V]^{(\ld)}$, lie in
every ideal $\I(I)$ and $\Pih(I)$. Thus the isomorphism of Theorem
\ref{thmDuality}, when restricted to the top degree, gives the
Poincare duality:
\[
H^k(S;\hh_0)=H^k(S;\hh_0\otimes\I^{(\ld)})\cong
H_{n-1-k}(S;\Pih^{(\ld)})=H_{n-1-k}(S;\ko).
\]
The restriction of the spectral sequence of Theorem
\ref{thmSpecSec} to the top degree gives the Zeeman--McCrory
spectral sequence in a similar way.
\end{rem}

%
%
%
%
%
%
%

\section{Proof of Theorem \ref{thmSpecSec}}\label{SecProof12}

The idea of proof is the following. We construct a filtered double
differential complex $\ca{X}_{k,l}$ and then play with various
spectral sequences converging to its total homology.

Before we proceed we need a small technical lemma. Let $J\in S$.
If $i$ is a vertex of $J$, we have a map $\eta_i\colon
\Pi_i\hookrightarrow \I(J)$, which includes the ideal $\Pi_i$
generated by a linear form $\omega_i$ into the ideal $\I(J)$
generated by a larger set of linear forms.

Consider the sequence of maps
\begin{equation}\label{eqTaylorLike}
0\leftarrow\I(J)\stackrel{\eta}{\leftarrow}
\bigoplus_{\substack{I,\dim I=0\\I\leqslant
J}}\Pi_I\stackrel{\xi}{\leftarrow}\bigoplus_{\substack{I,\dim
I=1\\I\leqslant
J}}\Pi_I\stackrel{\xi}{\leftarrow}\bigoplus_{\substack{I,\dim
I=2\\I\leqslant J}}\Pi_I\stackrel{\xi}{\leftarrow}\ldots
\end{equation}
where $\eta$ is the direct sum of the maps $\eta_i$ over
$i\in\ver(S)$, $i\leqslant J$; and $\xi$ is the direct sum of
inclusion maps $\Pi_I\hookrightarrow \Pi_{I'}$, each rectified by
the incidence sign $[I:I']$. The sign convention obviously implies
that \eqref{eqTaylorLike} is a differential complex. But what is
more important,

\begin{lemma}\label{lemmaTaylorLike}
The sequence \eqref{eqTaylorLike} is exact.
\end{lemma}

\begin{proof}
This is very similar to the Taylor resolution of monomial ideal in
commutative polynomial ring (or Koszul resolution), but our
situation is a bit different, since $\Pi_I$ are not free modules
over $\Lambda$. Anyway, the proof is similar to commutative case:
exactness of \eqref{eqTaylorLike} follows from inclusion-exclusion
principle. To make things precise (and also to tackle the case
$\ko=\Zo$) we proceed as follows.

By $\sta{\ko}$-condition, the subspace $\langle\omega_{j}\mid j\in
J\rangle$ is a direct summand in $V\cong \ko^{\ld}$. Let
$\{\nu_1,\ldots,\nu_{\ld}\}$ be a basis of $V$ such that its first
$|J|$ vectors are exactly $\omega_j$, $j\in J$. We simply identify
$J$ with the subset $\{1,\ldots,|J|\}\subseteq [\ld]$ by abuse of
notation. The module $\Lambda[V]$ splits in multidegree
components: $\Lambda=\bigoplus_{A\subseteq[\ld]}\Lambda_{A}$,
where $\Lambda_A$ is a 1-dimensional $\ko$-module generated by
$\bigwedge_{i\in A}\nu_i$. All modules and maps in
\eqref{eqTaylorLike} respect this splitting. Thus
\eqref{eqTaylorLike} can be written as
\begin{gather*}
0\longleftarrow \bigoplus_{A\cap J\neq\emptyset}\Lambda_A
\longleftarrow \bigoplus_{I\subseteq J,
|I|=1}\bigoplus_{A\supseteq I}\Lambda_A \longleftarrow
\bigoplus_{I\subseteq J, |I|=2} \bigoplus_{A\supseteq I}\Lambda_A
\longleftarrow \ldots,
\\
\bigoplus_{A, A\cap J\neq \emptyset}\left(0 \longleftarrow
\Lambda_A\longleftarrow\bigoplus_{I\subseteq A\cap J, |I|=1}
\Lambda_A \longleftarrow \bigoplus_{I\subseteq A\cap J, |I|=2}
\Lambda_A \longleftarrow\ldots \right).
\end{gather*}
For each $A$, the homology of the complex in brackets coincides
with $\Hr_*(\Delta_{A\cap J};\Lambda_A)\cong\Hr_*(\Delta_{A\cap
J};\ko)$, the reduced simplicial homology of the simplex on the
set $A\cap J\neq\emptyset$. Thus homology vanishes.
\end{proof}

Let us define a cosheaf $\ca{N}$ on $S$ taking values in graded
differential complexes. We set $\ca{N}(I)=C_*(G_I;\Pi_I)$, the
simplicial chains of the simplicial complex $G_I$. The
corestriction maps $\ca{N}(I>J)$ are naturally induced by
inclusions of faces $G_I\hookrightarrow G_J$ and inclusions of
coefficient modules $\Pih(I>J)\colon \Pi_I\hookrightarrow \Pi_J$.

The chain complex
\[
\ca{X}_{*,*}=(\Cc_*(S;\ca{N});d_H),\qquad \ca{X}_{k,l} =
\bigoplus_{I, \dim I=k}C_l(G_I;\Pi_I)
\]
is a double complex. It has the horizontal homological
differential $d_H\colon \ca{X}_{k,l}\to\ca{X}_{k-1,l}$
(sheaf-differential) and the vertical differential $d_V\colon
C_l(G_I;\Pi_I)\to C_{l-1}(G_I;\Pi_I)$ (inner differential). The
differentials commute, $d_Hd_V=d_Vd_H$, so we can form a totalized
differential complex
\[
\ca{X}_{j}=\bigoplus_{k+l=j}\ca{X}_{k,l},\qquad
\dtot=d_H+(-1)^kd_V\colon \ca{X}_j\to\ca{X}_{j-1}.
\]

\begin{lemma}\label{lemmaCosheafDoubleComp}
$H_k(\ca{X}, \dtot)\cong H_k(S;\Pih)$.
\end{lemma}

\begin{proof}
Consider the vertical spectral sequence \cite{McCl} converging to
$H_k(\ca{X}, \dtot)$:
\[
\E{V}^r_{*,*},\qquad
\dif{V}_r\colon\E{V}^r_{k,l}\to\E{V}^r_{k-r,l+r-1},
\]
which at first computes vertical homology, then horizontal. We
have
\[
\E{V}^1_{k,l}=\bigoplus_{I, \dim I=k} H_l(G_I;\Pi_I).
\]
Since $G_I$ is contractible, $H_l(G_I;\Pi_I)=0$ for $l\neq 0$ and
$H_0(G_I;\Pi_I)=\Pi_I$. Thus
\[
\E{V}^1_{k,l}=\begin{cases} \bigoplus_{\dim I=k}\Pi_I =
\Cc_k(S;\Pih), \mbox{ if }l=0;\\
0,\mbox{ otherwise.}
\end{cases}
\]
\[
\E{V}^2_{k,l}=\begin{cases} H_k(S;\Pih), \mbox{ if }l=0;\\
0,\mbox{ otherwise.}
\end{cases}
\]
The spectral sequence collapses at the second page, thus
$H_k(\ca{X}, \dtot)\cong H_k(S;\Pih)$.
\end{proof}

Our next goal is to compute the homology of totalization by first
computing the horizontal homology, then vertical. Recall that
$G_I$ is a simplicial subcomplex of $S'$, so the module
$\Cc_*(G_I;\Pi_I)$ is considered as the chain complex of the
constant cosheaf $\Pi_I$. Let the cosheaf $\prh{I}$ be the
corefinement of the sheaf $\I$ (recall this notion from subsection
\ref{SubsecCorefin}).

\begin{lemma}\label{lemmaCorefinement}
The sequence
\begin{equation}\label{eqTaylorLike2}
0\longleftarrow\Cc_*(S';\prh{I})\longleftarrow\bigoplus_{I, \dim
I=0}\Cc_*(G_I;\Pi_I)\longleftarrow\bigoplus_{I, \dim
I=1}\Cc_*(G_I;\Pi_I)\longleftarrow\ldots
\end{equation}
is exact.
\end{lemma}

\begin{proof}
Since all the maps $\Cc_*(G_I;\Pi_I)\rightarrow\Cc_*(G_I;\Pi_I)$
are induced by inclusions of simplicial subcomplexes, the sequence
\eqref{eqTaylorLike2} decomposes as the direct sum over all
simplices $\Delta=(I_1<\ldots<I_k)\in S'$:
\[
\bigoplus_{\Delta\in
S'}\left(0\longleftarrow\prh{I}(\Delta)\longleftarrow\bigoplus_{\substack{I,
\dim I=0\\\Delta\in
G_I}}\Pi_I\longleftarrow\bigoplus_{\substack{I, \dim
I=1\\\Delta\in G_I}}\Pi_I\longleftarrow\ldots \right)
\]
Since the condition $\Delta\in G_I$ is equivalent to $I_1\geqslant
I$, and by the definition of corefinement $\prh{I}$, the
expression in brackets is equal to
\[
0\longleftarrow\I(I_1)\longleftarrow \bigoplus_{\substack{I, \dim I=0\\
I\leqslant I_1}}\Pi_I \longleftarrow \bigoplus_{\substack{I, \dim I=1\\
I\leqslant I_1}}\Pi_I\longleftarrow\ldots
\]
This sequence is exact by Lemma \ref{lemmaTaylorLike}
\end{proof}

Let us return to the double complex $\ca{X}$ and consider its
horizontal spectral sequence
\[
\E{H}^r_{*,*},\qquad
\dif{H}_r\colon\E{H}^r_{k,l}\to\E{H}^r_{k+r-1,l-r}
\]
which computes horizontal homology first, then vertical.

\begin{lemma}\label{lemmaCorefinement2}
$H_l(\ca{X},\dtot)\cong H_l(S';\prh{I})$.
\end{lemma}

\begin{proof}
By Lemma \ref{lemmaCorefinement} the horizontal homology of
$\ca{X}$ vanishes except in degree $k=0$, where it is isomorphic
to $\Cc_*(S';\prh{I})$. Thus
\[
\E{H}^2_{k,l}\cong \begin{cases}H_l(S';\prh{I}), \mbox{ if } k=0;\\
0,\mbox{ otherwise.}
\end{cases}
\]
The spectral sequence collapses and the statement follows.
\end{proof}

Finally, we make use of the coskeleton filtration on $S'$.

\begin{lemma}\label{lemmaCoskelFilt}
There exists a spectral sequence $E^r_{s,k}\Rightarrow
H_{s+k}(S';\prh{I})$, $d^r\colon E^r_{s,k}\to E^r_{s-r,k+r-1}$,
$E^2_{s,k}\cong H^{n-1-s}(S;\hh_k\otimes \I)$. This spectral
sequence respects the inner gradings on $\I$ and $\prh{I}$
\end{lemma}

\begin{proof}
Consider the spectral sequence associated with the coskeleton
filtration of $S'$ for the coefficient system $\prh{I}$:
\[
E^r_{s,k}\Rightarrow H_{s+k}(S';\prh{I}),\quad d^r\colon
E^r_{s,k}\to E^r_{s-r,k+r-1},
\]
\[
E^1_{s,k}\cong H_{s+k}(S_s,S_{s-1};\prh{I}).
\]
We have
\[
E^1_{s,k}\cong H_{s+k}(S_s,S_{s-1};\prh{I}) = \bigoplus_{I,\dim
G_I=s}H_{s+k}(G_I,\dd G_I;\prh{I})= \bigoplus_{I,\dim
G_I=s}\new{I}_k(I)
\]
by Lemma \ref{lemmaTermSplitGenSheaf}. Since the values of $\I$
are torsion-free, Lemma \ref{lemmaNewSheafIsOld} implies
\[
\bigoplus_{I,\dim G_I=s}\new{I}_k(I)\cong\bigoplus_{I,\dim
G_I=s}(\I\otimes\hh_k)(I) = \Cc^{n-1-s}(S;\I\otimes\hh_k).
\]
Therefore,
\[
E^2_{s,k}\cong H^{n-1-s}(S;\I\otimes\hh_k)
\]
which proves the statement.
\end{proof}

The combination of lemmas \ref{lemmaCosheafDoubleComp},
\ref{lemmaCorefinement2}, and \ref{lemmaCoskelFilt} proves Theorem
\ref{thmSpecSec}.

%
%
%
%
%
%
%

\section{Manifolds with locally standard torus
actions}\label{SecManifolds}

%
%

\subsection{Orbit spaces}

Let $T^n$ be a compact $n$-dimensional torus. The standard
representation of $T^n$ is a representation
$T^n\curvearrowright\Co^n$ by coordinate-wise rotations, i.e.
\[(t_1,\ldots,t_n)\cdot(z_1,\ldots,z_n)=(t_1z_1,\ldots,t_nz_n),\]
for $z_i,t_i\in \Co$, $|t_i|=1$. An action of $T^n$ on a (compact
connected smooth) manifold $M^{2n}$ is called \emph{locally
standard}, if $M$ has an atlas of standard charts, each isomorphic
to a subset of the standard representation. More precisely, a
standard chart on $M$ is a triple $(U,f,\psi)$, where $U\subset M$
is a $T^n$-invariant open subset, $\psi$ is an automorphism of
$T^n$, and $f$ is a $\psi$-equivariant homeomorphism $f\colon U\to
W$ onto a $T^n$-invariant open subset $W\subset\Co^n$ (i.e.
$f(t\cdot y) = \psi(t)\cdot f(y)$ for all $t\in T^n$, $y\in U$).

The orbit space $\Co^n/T^n$ of the standard representation is the
nonnegative cone $\Ro_{\geqslant}^n=\{x\in \Ro^n\mid x_i\geqslant
0\}$. Thus an orbit space of a locally standard action obtains the
structure of compact connected $n$-dimensional manifold with
corners. Recall that a manifold with corners is a topological
space locally modeled by open subsets of $\Ro_{\geqslant}^n$ with
the combinatorial stratification induced from the face structure
of $\Ro_{\geqslant}^n$ (details relevant to the study of torus
actions can be found in \cite{BPnew} or \cite{Yo}).

%
%

\subsection{Characteristic functions}

Let $Q=M/T^n$ be the orbit space of a locally standard action. Let
$\fac(Q)$ denote the set of facets (i.e. faces of codimension
$1$). Every face $F$ of codimension $k$ lies in exactly $k$
distinct facets of $Q$ (such manifolds with corners are called
\emph{nice} in \cite{MasPan} or \emph{manifolds with faces}
elsewhere). Consider the set $S_Q$ of all faces of $Q$, including
$Q$ itself, and define the order on $S_Q$ by reversed inclusion.
Since $Q$ is nice, $S_Q$ is a simplicial poset. The minimal
element of $S_Q$ is the maximal face, that is the space $Q$
itself. The facets of $Q$ correspond to the vertices of $S_Q$. For
convenience we denote abstract elements of $S_Q$ by $I$,$J$, etc.
and the corresponding faces of $Q$ will be denoted by $F_I$,
$F_J$, etc.

If $F\in \fac(Q)$ and $x$ is a point from interior of $F$, then
the stabilizer of $x$, denoted by $\lambda(F)$, is a 1-dimensional
toric subgroup in $T^n$. If $F_I$ is a codimension $k$ face of
$Q$, contained in the facets $F_1,\ldots,F_k\in \fac(Q)$, then the
stabilizer of an orbit $x\in F_I^{\circ}$ is the $k$-dimensional
torus $T_I=\lambda(F_1)\times\ldots\times\lambda(F_k)\subset T^n$,
where the product is free inside $T^n$. This puts a specific
restriction on subgroups $\lambda(F)$, $F\in \fac(Q)$. In general,
the map
\begin{equation}\label{eqCharFunc}
\lambda\colon \fac(Q)\to \{\mbox{1-dimensional toric subgroups of
} T^n\}
\end{equation}
is called a \emph{characteristic function}, if, whenever the
facets $F_1,\ldots,F_k$ have nonempty intersection, the map
\[
\lambda(F_1)\times\ldots\times\lambda(F_k)\to T^n,
\]
induced by inclusions $\lambda(F_i)\hookrightarrow T^n$, is
injective and splits. This condition is called $(\ast)$-condition.
Notice, that $F_1,\ldots,F_k$ have nonempty intersection whenever
the corresponding vertices of $S_Q$ are the vertices of some
simplex.

From the $(\ast)$-condition follows that the map
\begin{equation}\label{eqHomolSplits}
H_1(\lambda(F_1)\times\ldots\times\lambda(F_k);\ko)\to
H_1(T^n;\ko)
\end{equation}
is also injective and splits for any ground ring $\ko$. Thus the
homology classes $\omega_1,\ldots,\omega_k$ of subgroups
$\lambda(F_1),\ldots,\lambda(F_k)$ freely span a direct summand in
$H_1(T^n;\ko)$. This motivates the definition of homological
characteristic function given in Section~\ref{SecAlgebras}.
Surely, the exterior algebra $\Lambda[V]$ generated by a
$\ko$-module $V$ has a clear meaning of the whole homology algebra
of a torus: $\Lambda[H_1(T^n;\ko)]\cong H_*(T^n;\ko)$.

If the function \eqref{eqCharFunc} satisfies \eqref{eqHomolSplits}
for some specific ground ring $\ko$, we say that $\lambda$
satisfies $\sta{\ko}$-condition. It is easy to see that the
topological $(*)$-condition is equivalent to $\sta{\Zo}$, and that
$\sta{\Zo}$ implies $\sta{\ko}$ for any $\ko$.

%
%

\subsection{Model spaces}

Let $M$ be a manifold with locally standard action and $\mu\colon
M\to Q$ be the projection to the orbit space. The free part of the
action has the form $\mu|_{Q^{\circ}}\colon \mu^{-1}(Q^{\circ})\to
Q^{\circ}$, where $Q^{\circ}=Q\setminus \dd Q$ is the interior of
the manifold with corners. The free part is a principal torus
bundle over $Q^{\circ}$. It can be uniquely extended over $Q$ and
defines a principal $T^n$-bundle $\rho\colon Y\to Q$.

Therefore any manifold with locally standard action determines
three objects: the nice manifold with corners $Q$, the principal
torus bundle $\rho\colon Y\to Q$, and the characteristic function
$\lambda$. One can recover the manifold $M$ from these data by the
following standard construction.

\begin{con}[Model space]\label{conModelSpace}
Let $\rho\colon Y\to Q$ be a principal $T^n$-bundle over a nice
manifold with corners $Q$ and $\lambda$ be a characteristic
function on $\fac(Q)$. Consider the space $X\eqd Y/\sim$, where
$y_1\sim y_2$ if and only if $\rho(y_1)=\rho(y_2)\in F_I^{\circ}$
for some face $F_I$ of $Q$, and $y_1,y_2$ lie in the same
$T_I$-orbit of the action. There exists a natural
$T^n$-equivariant map $f\colon Y\to X$.
\end{con}

Every manifold with locally standard torus action is equivariantly
homeomorphic to its model (\cite[Cor.2]{Yo}), so in the following
we will work with $X$ instead of $M$.

%
%

\subsection{Filtrations}

Since $Q$ is a manifold with corners, there is a natural
filtration on $Q$:
\begin{equation}\label{eqFiltrQ}
\emptyset=Q_{-1}\subset Q_0\subset Q_1\subset\ldots\subset
Q_{n-1}=\dd Q\subset Q=Q_n,
\end{equation}
where $Q_i$ is the union of all faces of dimension $\leqslant i$.
It lifts to the $T^n$-invariant filtration on $Y$:
\begin{equation}\label{eqFiltrY}
\emptyset=Y_{-1}\subset Y_0\subset Y_1\subset\ldots\subset
Y_{n-1}\subset Y_n=Y,
\end{equation}
where $Y_i=\rho^{-1}(Q_i)$. This in turn descends to the
filtration on $X$:
\begin{equation}\label{eqFiltrX}
\emptyset=X_{-1}\subset X_0\subset X_1\subset\ldots\subset
X_{n-1}\subset X_n=X,
\end{equation}
$X_i=f(Y_i)$. It is easily seen that \eqref{eqFiltrX} is the
filtration of $X$ by orbit types, i.e. $X_i$ is the union of all
orbits of dimension at most $i$. We have $\dim X_i=2i$. The maps
$\mu\colon X\to Q$, $\rho\colon Y\to Q$ and $f\colon Y\to X$
preserve the filtrations.

The filtrations give rise to homological spectral sequences.

\begin{gather}
\E{Q}^1_{p,q}=H_{p+q}(Q_p,Q_{p-1}) \Rightarrow H_{p+q}(Q),\qquad
\dif{Q}^r\colon \E{Q}^r_{*,*}\to\E{Q}^r_{*-r,*+r-1}\\
\E{Y}^1_{p,q}\cong H_{p+q}(Y_p,Y_{p-1}) \Rightarrow
H_{p+q}(Y),\qquad \dif{Y}^r\colon \E{Y}^r_{*,*}\to\E{Y}^r_{*-r,*+r-1}\\
\E{X}^1_{p,q}\cong H_{p+q}(X_p,X_{p-1}) \Rightarrow
H_{p+q}(X),\quad \dif{X}^r\colon
\E{X}^r_{*,*}\to\E{X}^r_{*-r,*+r-1}.
\end{gather}

In the following we also need the spectral sequence associated to
the filtration of $Q$ truncated at $Q_{n-1}=\dd Q$:
\begin{equation}\label{eqFiltrDQ}
\emptyset=Q_{-1}\subset Q_0\subset Q_1\subset\ldots\subset
Q_{n-1}=\dd Q,
\end{equation}
\[
\E{\dd Q}^1_{p,q}=\begin{cases} H_{p+q}(Q_p,Q_{p-1})\mbox{ for }
p<n,\\ 0,\mbox{ for } p=n\end{cases} \Rightarrow H_{p+q}(\dd Q).
\]

Note that $\E{X}^1_{p,q}=0$ for $q>p$ by dimensional reasons. The
map $f\colon Y\to X$ induces the map of spectral sequences
\[
f_*^r\colon \E{Y}^r_{p,q}\to\E{X}^r_{p,q}.
\]

The main topological result of this paper is the following

\begin{thm}\label{thmToricSpecSec}
If $Q$ is orientable and all proper faces of $Q$ are acyclic over
$\ko$, then the map $f_*^2\colon\E{Y}^2_{p,q}\to \E{X}^2_{p,q}$ is
an isomorphism for $q<p$ or $q=p=n$, and injective for $q=p<n$.
\end{thm}

%
%
%
%
%
%
%

\section{Proof of Theorem \ref{thmToricSpecSec}}\label{SecProof3}

At first we prove a technical lemma which is extremely useful when
one passes from the topology of $Q$ to the topology of its
underlying simplicial poset $S_Q$.

For a poset $S$ consider a space $P=\cone|S|$. A coskeleton
filtration of $S$ extends to the coskeleton filtration of $P$:
\[
|S_0|\subset\ldots\subset|S_{n-1}|=|S|\subset P,
\]
and the corresponding homological spectral sequence is denoted
$\E{P}^*_{*,*}$.

For convenience we introduce the following definition.

\begin{defin}
An oriented manifold with corners $Q$ is called Buchsbaum if all
its proper faces are acyclic. If $Q$ is Buchsbaum and $Q$ itself
is acyclic, then $Q$ is called Cohen--Macaulay. As usual, all
notions depend on the ground ring $\ko$.
\end{defin}

Any face $G$ of Buchsbaum manifold with corners $Q$ is an
orientable manifold with corners. The acyclicity of $G$ implies
that $H_j(G,\dd G)=0$ for $j\neq \dim G$ and $H_{\dim G}(G,\dd
G)\cong \ko$ by Poincare--Lefschetz duality.

\begin{lemma}\label{lemmaUniverPosets}\mbox{}

$(1)_n$ Let $Q$ be a Buchsbaum manifold with corners, $\dim Q =
n$, $S_Q$ be its underlying poset, and $P=\cone(|S_Q|)$. Then
there exists a face-preserving map $\varphi\colon Q\to P$ which
induces the identity isomorphism of posets of faces and the
isomorphism of the spectral sequences $\varphi_*\colon \E{\dd
Q}^r_{*,*}\stackrel{\cong}{\to}\E{S}^r_{*,*}$ for $r\geqslant 1$.

$(2)_n$ If $Q$ is Cohen--Macaulay of dimension $n$, then $\varphi$
induces the isomorphism of spectral sequences $\varphi_*\colon
\E{Q}^r_{*,*}\stackrel{\cong}{\to}\E{P}^r_{*,*}$ for $r\geqslant
1$.
\end{lemma}

\begin{proof}
A map $\varphi$ is constructed inductively. $0$-skeleta of $Q$ and
$P$ are naturally identified since both correspond to the set of
maximal simplices of $S$. There always exists an extension of
$\varphi$ to higher-dimensional faces since all faces of $P$ are
cones. The statement is proved by the following scheme of
induction: $(2)_{\leqslant n-1}\Rightarrow (1)_n\Rightarrow
(2)_n$. The case $n=0$ is clear.

Let us prove the implication $(1)_n\Rightarrow (2)_n$. The map
$\varphi$ induces the homomorphism of the long exact sequences:
\[
\xymatrix{ \Hr_*(\dd
Q)\ar@{->}[r]\ar@{->}[d]&\Hr_*(Q)\ar@{->}[r]\ar@{->}[d]& H_*(Q,\dd
Q)\ar@{->}[r]\ar@{->}[d]&\Hr_{*-1}(\dd
Q)\ar@{->}[r]\ar@{->}[d]&\Hr_{*-1}(Q)\ar@{->}[d]\\
\Hr_*(\dd P)\ar@{->}[r]&\Hr_*(P)\ar@{->}[r]& H_*(P,\dd
P)\ar@{->}[r]&\Hr_{*-1}(\dd P)\ar@{->}[r]&\Hr_{*-1}(P) }
\]
The maps $\Hr_*(Q)\to\Hr_*(P)$ are isomorphisms since both groups
are trivial. The maps $\Hr_*(\dd Q)\to\Hr_*(\dd P)$ are
isomorphisms, since $\E{\dd Q}\conviso H_*(\dd Q)$, $\E{\dd
P}\conviso H_*(\dd P)$ and the spectral sequences are isomorphic
by $(1)_n$. Five lemma shows that $\varphi_*\colon\E{Q}^1_{n,*}\to
\E{P}^1_{n,*}$ is an isomorphism as well. This proves $(2)_n$.

Now we prove the implication $(2)_{\leqslant n-1}\Rightarrow
(1)_n$. Let $F_I$ be faces of $Q$ and $G_I$ faces of $P$. All
proper faces of $Q$ are Cohen--Macaulay of dimension $\leqslant
n-1$. Thus $(2)_{\leqslant n-1}$ implies the isomorphisms
$H_*(F_I,\dd F_I)\to H_*(G_I,\dd G_I)$ which sum together to the
isomorphism $\varphi_*\colon\E{\dd
Q}^1_{*,*}\stackrel{\cong}{\to}\E{\dd P}^1_{*,*}$.
\end{proof}

\begin{cor}\label{corQbuchSbuch}
If $Q$ is a Buchsbaum manifold with corners, then $S_Q$ is
Buchsbaum. Moreover, in this case $S_Q$ is a homology manifold. If
$Q$ is Cohen--Macaulay, then $S_Q$ is a homology sphere.
\end{cor}

From now on we suppose that $Q$ is Buchsbaum, as stated in the
condition of Theorem \ref{thmToricSpecSec}. Thus $S_Q$ is
Buchsbaum as well.

Let us return to the spaces $Y$ and $X$ over $Q$. As before, let
$F_I$ be the face of $Q$ corresponding to $I\in S_Q$. Let
$Y_I=\rho^{-1}(F_I)$ and $X_I=f(Y_I)$ be the corresponding subsets
of $Y$ and $X$ respectively. Actually, $X_I\subset X$ is a closed
submanifold of dimension $2\dim F_I$, called a face submanifold.
We set $\dd Y_I=\rho^{-1}(\dd F_I)$ and $\dd X_I=f(\dd Y_I)$ (the
set $\dd X_I$ does not have the meaning of a boundary in a
topological sense, this is just a conventional notation). Note
that $Y_{\minel}=Y$ and $X_{\minel}=X$.

We have
\[
\E{Y}^1_{p,q}\cong H_{p+q}(Y_p,Y_{p-1})\cong
\bigoplus_{|I|=n-p}H_{p+q}(Y_I,\dd Y_I)
\]
\[
\E{X}^1_{p,q}\cong H_{p+q}(X_p,X_{p-1})\cong
\bigoplus_{|I|=n-p}H_{p+q}(X_I,\dd X_I)
\]

\begin{rem}\label{remXYsameForPeqN}
The map $f_*^1\colon \E{Y}^1_{n,q}\to \E{X}^1_{n,q}$, which
coincides with $f_*\colon H_*(Y,\dd Y)\to H_*(X,\dd X)$, is an
isomorphism since the identification $\sim$ of Construction
\ref{conModelSpace} touches only the boundary $\dd Y$, thus $Y/\dd
Y\cong X/\dd X$.
\end{rem}

The space $Y_I$ is a principal $T^n$-bundle over $Q_I$. For each
$I\in S\setminus\minel$, the face $Q_I$ is acyclic. Thus there
exists a trivialization $Y_I\cong Q_I\times T^n$ and we have
\begin{equation}\label{eqStructHomolY}
H_{p+q}(Y_I,\dd Y_I)\cong \bigoplus_{i+j=p+q} H_i(F_I,\dd
F_I)\otimes H_j(T^n)\cong H_q(T^n),
\end{equation}
(the groups $H_i(F_I,\dd F_I)$ vanish for $i\neq p$, and
$H_p(F_I,\dd F_I)\cong \ko$). Similarly, for $X$ we have the
identification
\[
H_*(X_I,\dd X_I)\cong H_*(F_I\times T^n/T_I, \dd F_I\times
T^n/T_I),
\]
thus
\begin{equation}\label{eqStructHomolX}
H_{p+q}(X_I,\dd X_I)\cong H_q(T^n/T_I).
\end{equation}
Consider the graded sheaf $\hh_q^Y$ on $S_Q$ which takes the value
$H_{p+q}(Y_I,\dd Y_I)$ on each $I\in S_Q$ (including $I=\minel$)
with the restriction maps extracted from the differential
$\dif{Y}^1$ similar to Construction \ref{conSheafOnS}. By
\eqref{eqStructHomolY}, the truncated part
$\hht_*^Y=\bigoplus_q\hht_q^Y$ (see Remark
\ref{remTruncateTheSheaf}) is the locally constant sheaf $\La$
valued by exterior algebras.

Similarly, we can define a graded sheaf $\hh_q^X$ on $S_Q$ which
takes the value $H_{p+q}(X_I,\dd X_I)$ on $I\in S$. Its truncated
part $\hht_*^X=\bigoplus_q\hht_q^X$ is the sheaf of quotient
algebras $\La/\I$ according to \eqref{eqStructHomolX}. Indeed, it
is easily seen that the homology algebra $H_*(T^n/T_I)$ is
naturally isomorphic to the quotient of $H_*(T^n)/\I(I)$, where
$\I(I)$ is the ideal generated by the subspace $H_1(T_I)\subset
H_1(T^n)$.

The map $f_*^1\colon \E{Y}^1_{*,*}\to \E{X}^1_{*,*}$ is equal to
the map $f_*\colon \Cc^*(S;\hh_*^Y)\to \Cc^*(S;\hh_*^X)$. This
last map coincides with $f_*\colon \Cc^*(S;\La)\to
\Cc^*(S;\La/\I)$ away from $\minel$.

\begin{lemma}\label{lemmaIshortNonTrunc}
There exists a short exact sequence of graded sheaves
\[
0\rightarrow \I\rightarrow \hh^Y \rightarrow \hh^X \rightarrow 0.
\]
\end{lemma}

\begin{proof}
This follows from the diagram
\[
\xymatrix{
&&0\ar@{->}[d]&0\ar@{->}[d]&\\
0\ar@{->}[r]&\I\ar@{^{(}->}[r]\ar@{=}[d]&
\hht^Y\ar@{->>}[r]\ar@{^{(}->}[d]&\hht^X\ar@{->}[r]\ar@{^{(}->}[d]&0\\
0\ar@{->}[r]&\I\ar@{^{(}->}[r]&\hh^Y\ar@{->>}[r]\ar@{->>}[d]&
\hh^X\ar@{->}[r]\ar@{->>}[d]&0\\
&&\hh^Y/\hht^Y\ar@{->}[r]^{\cong}\ar@{->}[d]&\hh^X/\hht^X\ar@{->}[d]&\\
&&0&0& }
\]
in which all vertical and two horizontal lines are exact. The
lower sheaves are concentrated in $\minel\in S_Q$ and the graded
isomorphism between them is due to Remark~\ref{remXYsameForPeqN}.
\end{proof}

Finally, the short exact sequence of Lemma
\ref{lemmaIshortNonTrunc} induces the long exact sequence in sheaf
cohomology:
\begin{equation}\label{eqLongExact}
\rightarrow H^{i-1}(S_Q;\I^{(q)}) \rightarrow H^{i-1}(S_Q;\hh^Y_q)
\stackrel{f^2_*}{\longrightarrow} H^{i-1}(S_Q;\hh^X_q)
\longrightarrow H^{i}(S_Q;\I^{(q)})\rightarrow
\end{equation}

The poset $S_Q$ is a homology manifold. Thus its structure sheaf
is constant: $\hh_0\cong \ko$. Corollary \ref{corKeyCorol} implies
that the groups $H^i(S_Q;\I^{(q)})$ vanish for $i\leqslant n-1-q$.
From the long exact sequence \eqref{eqLongExact} we can see that
the map
\[
f_*\colon H^{i-1}(S_Q;\hh^Y_q)\to H^{i-1}(S_Q;\hh^X_q)
\]
is an isomorphism for $i\leqslant n-1-q$ and injective for
$i=n-q$. This map coincides with
\[
f^2_*\colon \E{Y}^2_{n-i,q} \to \E{X}^2_{n-i,q}.
\]
The change of indices $p=n-i$ concludes the proof of Theorem
\ref{thmToricSpecSec}.

\begin{rem}
Note, that the similar argument proves that the map $f^*\colon
\E{\dd Y}^2_{p,q}\to \E{\dd X}^2_{p,q}$ is an isomorphism for
$p>q$ and injective for $p=q$.
\end{rem}

\section*{Acknowledgements}
I am grateful to prof. Mikiya Masuda for his hospitality and for
the wonderful environment surrounding me in Osaka City University.
The problem of computing the cohomology ring of toric origami
manifolds, on which we worked since 2013, was a great motivation
for the current research (and served as a good setting to test
general hypotheses).


\end{document}